\documentclass{siamart1116}
\usepackage[utf8]{inputenc}
\usepackage{amsfonts}

\usepackage{todonotes}

\usepackage{ulem}

\renewcommand{\d}{\mathrm{d}}
\newcommand{\R}{\mathbb R}
\newcommand{\N}{\mathbb N}
\newcommand{\E}{\mathbb E}
\newcommand{\eps}{\varepsilon}
\renewcommand{\P}{\mathbb P}
\newcommand{\Res}{\operatorname{Res}}%
 \newcommand{\vb}{{\bar{v}}}

\newcommand{\rev}[1]{{{#1}}}
\newcommand{\revcls}[1]{{{#1}}}
\renewcommand{\sout}[1]{}

\newsiamremark{rem}{Remark}
\newsiamremark{calc}{Calculation}

\title{A strongly convergent numerical scheme from \revcls{ensemble Kalman inversion} \rev{\sout{ continuum analysis}}}

\author{ Bl\"omker, Dirk\thanks{Universit\"at Augsburg, \texttt{dirk.bloemker@math.uni-augsburg.de}} \and
  Schillings, Claudia\thanks{Universit\"at Mannheim, \texttt{c.schillings@uni-mannheim.de}}
  \and Wacker, Philipp\thanks{Universit\"at Erlangen, \texttt{phkwacker@gmail.com}}
  }

\begin{document}
\maketitle
\begin{abstract}
The Ensemble Kalman methodology in an inverse problems setting can be viewed \revcls{\sout{-- when constructed in a sequential Monte-Carlo-like manner --}} {as an iterative scheme, }
which  is a weakly tamed discretization scheme 
for a certain stochastic differential equation (SDE). \rev{\sout{ The continuous time limit of the method,
formally derived by Schillings and Stuart, allows to establish a well-posedness and convergence analysis in the inverse setting.}} 
Assuming a suitable approximation result,  
dynamical properties of the SDE can be rigorously pulled back via the discrete scheme to the original \revcls{Ensemble Kalman inversion}. 

\rev{The results of this} paper make a step towards closing the gap of the missing approximation result by proving a strong convergence result \rev{in a simplified model of a scalar stochastic differential equation}. 
We focus here on a \rev{toy model} with similar properties than the one arising in the \rev{context of} Ensemble Kalman filter. 
The proposed model can be interpreted as a single particle filter for a linear map and thus forms the basis for further analysis.
The difficulty in the analysis arises from the formally derived limiting SDE with non-globally Lipschitz continuous nonlinearities\rev{ both in the drift and in the diffusion.} 
\rev{Here the standard Euler-Maruyama scheme might fail to provi\revcls{d}e a strongly convergent numerical scheme and \revcls{taming} \revcls{is} necessary. 
In contrast to the strong \revcls{taming} usually used, the method presented here provides a weaker form of \revcls{taming}.} 

We present a strong convergence analysis by first proving convergence on a domain of high probability \rev{by using a cut-off or localisation}, which then leads, 
combined with bounds on moments for both the SDE and the numerical scheme, by a bootstrapping argument to \rev{strong convergence}.
\end{abstract}

%%%%%%%%%%%%%%%%%%%%%%%%%5

\section{Introduction}

We consider \rev{throughout the whole paper} the following \rev{simplified} {model given by the one-dimensional SDE}
\begin{equation}\label{eq:du1}
\d u = -u^3 \d t + u^2 \d W\;.
\end{equation}
For this equation, we will prove that it is strongly approximated for $t\in [0,T]$ by the numerical scheme arising in the context of 
\rev{Ensemble Kalman filter (EnKF)} \revcls{for inverse problems}
\begin{equation} \label{eq:uScheme}
u_{n+1} = u_n - h\cdot\frac{u_n^3}{1+h\cdot u_n^2} + \frac{u_n^2}{1+h\cdot u_n^2}\cdot \Delta W_{n+1}.
\end{equation}
with $n=0,\ldots, T/h$, where $u_n$ is supposed to correspond to $u(t_n)$ with $t_n = n\cdot h$ in the 
\rev{strong sense at a single fixed time, 
and in a uniform strong sense, where moments of the uniform error in time are taken.}

We are also able to extract a rate of convergence for both \rev{the strong and the uniform strong error.} 
\rev{\sout{with a better rate for the weaker notion in %\eqref{eq:weak_strongconvergence_continuous}
.}} 
The exact statement of the rate can be found in our main result {theorem \ref{lem:main}.}
\rev{In Section \ref{sec:mot} we will comment in detail how this approximation scheme arises as a toy model in the theory of EnKF.}

Note that a convergence result like the one above does not hold if we exchange the EnKF numerical scheme $u_n$ by the naive Euler-Maruyama discretization scheme
\begin{equation}\label{eq:EMScheme}
\tilde u_{n+1} = \tilde u_n - h\cdot \tilde u_n^3 + \tilde u_n^2 \cdot \Delta W_{n+1}
\end{equation}
as was shown by Hutzenthaler, Jentzen and Kloeden \cite{hutzenthaler2011strong}. 
In this setting we only obtain pathwise convergence but it can even be proven that the Euler-Maruyama discretization diverges strongly.
This is due to a exponentially rare (in $h$) family of events (paths of $\tilde u_n$) which grows bi-exponentially strongly (in $h$) 
while paths of $u$ stay near the origin and have $p$-moments at least up to $p<3$. Then, any moment of the difference between $u$ and $\tilde u_n$ will explode for $h\to 0$.

There has been significant progress in the field of strongly convergent numerical schemes for SDEs with non-global Lipschitz-continuous nonlinearities.
Standard references for numerical methods for SDEs, like \cite{kloedennumerical} and {\cite{mao2007stochastic,lord2014introduction}}, show strong convergence of the Euler-Maruyama method only for globally Lipschitz-continuous drift and diffusion terms. 
Higham, Mao and Stuart \cite{higham2002strong} proved a conditional result about strong convergence of the Euler-Maruyama discretization for non-globally Lipschitz SDEs given that moments of both solutions and the discretization stay bounded. 
This means that the question of strong convergence was replaced by the question whether moments of the numerical scheme stay bounded, but Hutzenthaler, 
Jentzen and Kloeden \cite{hutzenthaler2011strong} answered the latter to the negative, even proving that moments of the Euler-Maruyama scheme always explode in finite time if either drift or diffusion term are not globally Lipschitz. 
Instead, they proposed a slight modification of the EM method, the so-called ``tamed'' EM method (and implicit equivalents)
in {\cite{hutzenthaler2012strong,hutzenthaler2015numerical}}. 

The numerical scheme \rev{\eqref{eq:uScheme} arising in the EnKF analysis bears {resemblance} to the ``tamed'' methods used throughout the literature.
In case only the drift is non-globally Lipschitz a similar idea to \eqref{eq:uScheme} was used already in \cite{hutzenthaler2012strong}, 
but there the drift-tamed nonlinearity is strictly bounded by one,
while in EnKF it is still allowed to grow linearly.
In increment-tamed schemes (see Hutzenthaler and Jentzen \cite{hutzenthaler2015numerical}) even the whole increment on the right-hand-side of \eqref{eq:EMScheme} is cut to have modulus of at most one.}

Thus the differences \rev{to other tamed schemes} are still too large to carry out their analysis analogously.
{In particular, we cannot choose an arbitrary scheme with the right taming for a given SDE, 
but we face a specific iterative scheme directly given by the EnKF.
We emphasize that even the increment-tamed scheme \rev{\sout{ by Hutzenthaler, Jentzen and Kloeden}} cannot provably approximate the SDE 
$$
\d u = -u^3 \d t + u^2 \d W\,,
$$ 
which we focus on in this paper.}
This is easily checked by {corollary 3.18 in} \cite{hutzenthaler2015numerical}. 
For $\mu(x) = -x^3$ and $\sigma(x) = x^2$, we {cannot} find a positive number $q$ such that the conditions for strong convergence are met.
\rev{
\begin{rem}
We want to point out again that the goal of this manuscript is not to propose a new class of numerical discretization schemes for SDEs, 
although further research might reveal whether this is possible. 
Rather, we are initially ``stuck'' with a iterative scheme which has a formal continuous form (the SDE). We will show that the formal continuous limit 
is rigorous using the strong convergence notion. We believe that our results should be extendable in some situations to the case of  multivariate $u\in\R^d$
\begin{displaymath}
\d u =  f(u) \d t + \sigma(u) \d B
\end{displaymath}
where $B$ is a Brownian motion in $\R^d$ and $f$ and $\sigma$ are non-linearities which are only locally Lipschitz, although we haven't checked this properly. Especially, the correct \revcls{taming} of drift and diffusion that is weaker than the usual \revcls{taming} in the literature, 
but still preserves all the properties that allow for moment bounds for the numerical scheme and the SDE, like one-sided Lipschitz-properties, is open.
\end{rem}}

Our method of proof in showing strong convergence is similar in spirit to the method deployed in the first chapter of \cite{hutzenthaler2015numerical},
although in a continuous setting instead for a discrete process. 
We show strong convergence on a domain of high probability (governed by a stopping time in our case instead of restricting to events of high probability in the method of Hutzenthaler and Jentzen) 
and combined with bounds on moments for both the SDE and the numerical scheme, a bootstrapping argument facilitated by H\"older's inequality proves convergence on the whole domain. Our method also bears resemblance to the trick employed by Higham, Mao and Stuart \cite{higham2002strong}. 
{The key difference is that we need to avoid using Gronwall's lemma in its integrated form. Rather we exploit the dissipative structure of the SDE explicitly and
obtain in addition to the uniform bounds also bounds on time-averaged quantities. This helps us to eliminate the need for high-order moments.} This is reflected in equation \eqref{eq:goodmoment}: The integral term originates as a negative term on the right hand side as a result of the stability of the SDE. 
The key results are lemmas \ref{lem:rate} and \ref{lem:rate_weak}.

The many-particle setting makes the analysis very complicated, 
so we will work with the {simplified model} \rev{\eqref{eq:du1}} in one dimension which still shows the main characteristics as the SDE and the numerical scheme in the EnKF setting. 
For this model we will show that the numerical scheme which arises from the EnKF iteration is in fact a strongly converging discretization.

{This is a surprising fact given that for the specific equation $\d u = -u^3 \d t + u^2 \d W$, both the Euler-Maruyama discretization and the increment-tamed scheme proposed by Hutzenthaler et al. do not share this property. We also note that non-strongly converging numerical schemes are really only an issue for SDEs with non-globally Lipschitz drift and/or diffusion terms, which is the case here.}

 \revcls{\sout{ The simplified model is a good starting point for the analysis and will form the basis for the more general setting, which will be subject of future work. Due to the subspace property of the EnKF, the dynamics can be described in the finite dimensional coordinate system of the particles, thus the analysis of the finite dimensional case can be straightfowardly generalized to the EnKF setting.
However, the interaction of the particles and the nonlinearity of the forward response operator results in more complicated nonlinearities of the drift and diffusion term, so that we anticipate the need of additional assumptions on their stability in the more general setting.}}

In order to prove strong convergence we consider \rev{\sout{a surrogate model (which can be interpreted as a Lipschitz-regularized version} an intermediate model which is the regularization of \eqref{eq:du1})}
\begin{equation}\label{eq:dv}
\d v = -\frac{v^3}{1+\eps\cdot v^2}\d t + \frac{v^2}{1+\eps\cdot v^2}\d W.
\end{equation}
Its Euler-Maruyama discretization is then
\begin{equation} \label{eq:scheme}
\begin{split}
v_{n+1} &= v_n + h\cdot \frac{-v_n^3}{1+\eps\cdot v_n^2}+\Delta W_n\cdot \frac{v_n^2}{1+\eps\cdot v_n^2}\\
 &= v_n + \int_{t_n}^{t_{n+1}}\frac{-v_n^3}{1+\eps\cdot v_n^2}\d t + \int_{t_n}^{t_{n+1}}\frac{v_n^2}{1+\eps\cdot v_n^2}\d W
 \end{split}
\end{equation}
and from standard theory we know that the discretization converges strongly to the SDE.
The main point is that for $h = \eps$ this numerical scheme is identical to the one we are interested in, i.e. $u_n$. We will set $\eps = h$ in what follows so we are actually considering a sequence of numerical schemes for which we want to prove that they approximate a sequence of SDEs (the family of paths of $v$ parameterized by $\eps$).
We set $f(v) = -\frac{v^3}{1+\eps\cdot v^2}$ and $\sigma(v) = \frac{v^2}{1+\eps\cdot v^2}$ and define the interpolation of $v_n$ as 
\begin{equation}
\bar v(t) = v_n + \int_{t_n}^t f(v_n)\d s + \int_{t_n}^{t}\sigma(v_n)\d W.
\end{equation}
Note that this is not actually an interpolation in the normal sense because it has ``wiggly'' paths from the stochastic integral.
We will show that
\begin{enumerate}
\item $v$ converges strongly to $u$. More precisely we prove
\begin{equation}
\sup_{t\in [0, T]} \lim_{\eps \to 0} \E |u(t)-v(t)|^\alpha = 0 \quad\text{for any }0 < \alpha < 3
\end{equation}  
and
\begin{equation}
\lim_{\eps \to 0} \E\sup_{t\in [0, T]}  |u(t)-v(t)|^\beta = 0 \quad\text{for any }0 < \beta < 3/2. 
\end{equation}
We also obtain rates of this convergence, see lemma \ref{lem:ratesuv}.
\item $\bar v$ and $v$ become arbitrarily close in the strong sense. More precisely, 
\begin{equation}\label{eq:approx_u_v}
 \lim_{h=\eps \to 0}\sup_{t\in[0, T]} \E |v(t)-\bar v(t)|^\alpha = 0\quad\text{for any }{0 < \alpha < 2}, 
\end{equation}
and also
 \begin{equation}\label{eq:strongapprox_u_v}
  \lim_{h=\eps \to 0}\E\sup_{t\in[0, T]} |v(t)-\bar v(t)|^\beta = 0\quad\text{for  all }0 < \beta < 1,
\end{equation}
again with rates as shown in  lemma \ref{lem:ratesvbarv}.
\end{enumerate}
This means that paths of the \rev{\sout{surrogate model} regularization} $v$ are arbitrarily well-approximated (strongly) by the numerical scheme.
By the triangle inequality we obtain our main result\rev{\sout{  
and  
}}, 
keeping in mind that $\bar v(t_n) = u_n$ for $\eps = h$. 

\rev{Key for our analysis are a-priori moment bounds on all quantities involved, i.e. $\sup_{t\in[0,T]} \E |w|^\kappa < C$ and $\E\sup_{t\in[0,T]} |w|^\xi$ for $w\in\{u, v, \bar v\}$ and $\kappa,\xi > 0$ chosen suitably (cf. lemma \ref{lem:rate}).}

\begin{rem}
Note that we are able to show
\begin{equation*} 
\lim_{\eps \to 0} \E\sup_{t\in [0, T]} |u(t)-v(t)|^\beta = 0 \quad\text{for }\beta \in (0, 3/2)\;,
\end{equation*}  
but a similar result for the difference $|v(t)-\bar v(t)|$ holds only for $\beta \in (0,1)$. {The difference $v-\vb$ is a similar bottleneck for the $\sup \E$-case, where we only achieve $0<\alpha < 2$ instead of $0 < \alpha < 3$ for the difference $u-v$.}
{In the case ``$\E\sup$'', we expect that the result can be shown to hold for $\beta\in (0, 3/2)$. However, the presented analysis does not include the uniform boundedness (in $\eps$) of
\begin{displaymath} \E\sup_{t\in[0,T]} |\bar v(t)|^\beta < C
\quad\text{or equivalently}\quad 
 \E \sup_{n=1,\ldots, T/h} |v_n|^\beta < C
 \end{displaymath}
for $\beta \in (1, 3/2)$.} \rev{\sout{{These bounds are the key for our analysis (cf. lemma \ref{lem:rate})}}}
\end{rem}
\begin{rem}
We took care to formulate ``$\bar v$ and $v$ become arbitrarily close''. 
We will from now on sometimes {use the intuitive notion of} ``$\bar v$ converges to $v$'' for simplicity, although this is {not rigorously defined as} $v$ varies with $\eps=h$ as well.
\end{rem}

%%%%%%%%%%%%%%%%%%%%%%%%%%%%%%
\section{Motivation from the EnKF and its continuum analysis}
\label{sec:mot}

The Ensemble Kalman filter (EnKF) was first {introduced by Evensen (\cite{evensen1994sequential}). 
Since then, it has been widely used in the context of data assimilation and recently also of inverse problems, 
see e.g. \cite{law2015data} for more details. Despite its great success story in various areas of application, 
the analysis of the EnKF is still in its infancy. Well-posedness and convergence results in the case of a fully observed system are discussed in \rev{\cite{kelly2014well,tongnonlinear,tonginfl,Kelly25082015}} and in \cite{2016arXiv161206065D} in the data assimilation context. The analysis of the large ensemble size limit can be found in \cite{mandel2011convergence,kwiatkowski2015convergence, gratton2014convergence}. The generalization to inverse problems is presented in \cite{iglesias2013ensemble}. Ernst, Sprungk 
and Starkloff \cite{ernst2015analysis} showed that the EnKF is not consistent with the Bayesian perspective in {the }nonlinear setting, but can be interpreted as a point estimator of the unknown parameters. 
Most of the analysis is confined to the large ensemble size limit, i.e. the case where the number of particles in the ensemble goes to infinity. However, the EnKF is usually used with a rather small ensemble size in practice. 
\revcls{This finite-ensemble scenario was analysed by Kelly et al. \cite{kelly2014well,Kelly25082015,tongnonlinear,tonginfl}. {Stannat \cite{stannat2011stability} analysed the behavior of the EnKF for the filtering problem of a partially observed solution of an SDE. This setting is different from ours in that the scaling of the data (necessarily) is completely different: In the mentioned paper, more frequent observations lead to more available data for filtering. In our scenario, the data is artifically augmented to ``live'' on a time-axis and is used repeatedly (and redundantly). To compensate for this redundancy, an artificial noise term is added to the data (cf. \eqref{eq:perturbedobservation}. For a more complete discussion of this, see \cite{schillings2016analysis,iglesias2013ensemble}.}

The authors presented an analysis of the long-time behavior and ergodicity of the ensemble Kalman filter with arbitrary ensemble size establishing time uniform bounds to control the filter divergence and ensuring in addition the existence of an invariant measure. In the linear Gaussian setting, Del Moral and Tugaut \cite{delmoral} investigated the convergence of the ensemble Kalman-Bucy filter and provided time uniform error estimates for the empirical mean and covariance.} 
Schillings and Stuart \cite{schillings2016analysis} \rev{(and prior to that, Bergemann and Reich in \cite{bergemann2012ensemble})} conducted a continuum limit analysis of the EnKF methodology, 
shedding light on the properties of the EnKF in the small ensemble size setting. 
They propose to interpret the Ensemble Kalman method as a numerical discretization scheme of  a continuous process driven by a (stochastic) differential equation by formally conducting the continuum limit. 
\revcls{In \cite{amezcua2014ensemble}, the behavior of deterministic EnKF limits has been studied to develop stable and efficient integration methods.}}
The purpose of this paper is to provide a step towards making this continuum limit rigorous (in a probabilistically strong meaning).
{In the following, we shortly introduce the Ensemble Kalman filter for inverse problems and refer to \cite{iglesias2013ensemble,schillings2016analysis} for a more detailed exposition.}

\paragraph{The use of the EnKF in Bayesian inverse problems}
The inverse problem is defined as follows:
The goal of computation is to recover the unknown parameters $u \in X$ from a noisy observation $y \in Y$ where
\begin{equation}
\label{eq:IP} y = G(u) + \eta 
\end{equation}
with the so-called forward response map $G:X\to Y$. We denote by $X$ and $Y$ separable Hilbert spaces and by the random variable $\eta$ the observational noise.
 
Inverse problems arise in a multitude of applications, for example in geosciences, medical imaging, reservoir modelling, astronomy and signal processing. A probabilistic approach has been undertaken first by Franklin \cite{franklin1970well}, Mandelbaum \cite{mandelbaum1984linear} and others, and a formulation of inverse problems in the context of Bayesian probability theory was done by Fitzpatrick \cite{fitzpatrick1991bayesian}. More recent literature about inverse problems in the Bayesian setting can be found in \cite{neubauer2008convergence}, \cite{stuart2010inverse}.%, \cite{helin2015maximum}.

In the following, we assume that the number of observations is finite, i.e. $Y = \R^K$ for some $K\in\N$, 
whereas the unknown is a distributed quantity, which is a typical setting for many applications mentioned above. 
\revcls{\sout{Thus, the inverse problem consists of inferring infinitely many parameters from a finite number of noisy observations.
This leads to an ill-posed problem, which needs to be regularized. We will focus here on the Bayesian approach, i.e. the inverse problem becomes well-posed by a prior distribution on the feasible values for $u$ and the data gets incorporated by a Bayesian update on this prior measure, 
yielding the posterior measure in result. }} Furthermore, we assume that the observational noise is Gaussian, i.e. $\eta\sim N(0,\Gamma)$ with a symmetric, positive definite matrix $\Gamma\in\mathbb R^{K\times K}$.
{We consider} the least squares ``error'' (or model-data misfit) functional
\begin{equation}
\label{eq:errorfunctional}
\Phi(u; y) = \frac{1}{2}\| \Gamma^{-1/2}(y-G(u))\|_Y^2,
\end{equation}
where $\Gamma$ normalizes the model-data misfit and is normally chosen as the covariance operator of the noise.
{Plain infimization of this cost functional is not feasible due to the ill-posedness of the problem.} For a given prior $u \sim \mu_0$, we derive the posterior measure $u|y \sim \mu$ where 
\begin{equation}\label{eq:posteriormeasure}
\mu(du) = \frac{1}{Z}\exp(-\Phi(u;y))\mu_0(du).
\end{equation}
\paragraph{Sequential version of the EnKF}
By interpolating the step $\mu_0 \rightsquigarrow \mu$ as $\mu_n(du)\propto \exp(-nh\Phi(u;y))\mu_0(du)$ with  a step size $h = 1/N$ and $n\in \{1,\ldots, N\}$, 
a finite sequence of measures is obtained with initial measure $\mu_0$ (the prior) and final measure $\mu_N = \mu$. 
{Note that, by introducing the artificial time step $h$, we recast the problem into a data assimilation problem by using the observational data sequentially. We account for the repeated use of the data by amplifying the noise covariance $1/h\cdot \Gamma$. 
It is important to keep in mind that the time we refer to in the following is the artificial time introduced by the sequence of intermediate measures $\mu_n$.}

{The Ensemble Kalman Filter approximates} each of those measures $\mu_n$ by a sum of Dirac masses (their centres form the ensemble of particles giving its name to the EnKF):
\begin{equation}\label{eq:DiracApprox}
\mu_n \approx \nu_n = \frac{1}{J}\sum_{j=1}^J \delta_{u_n^{(j)}}\;.
\end{equation}
Then the problem of mapping $\mu_n \mapsto \mu_{n+1}$ reduces to mapping the particles $u_n^{(j)}$ in ``time'' step $n\mapsto n+1$. 
 
{The EnKF chooses a linear transformation of the particles such that the mean and covariance are consistent updates with the Kalman filter. Various variants of the EnKF realizing different transformations of the particles exist, see e.g. \cite{reich2015probabilistic} for more details. We focus here on the oldest transformation, the EnKF with perturbed observations
leading to the iteration of the form
\begin{equation}
\label{eq:EnKFDynamics}
u_{n+1}^{(j)} = u_n^{(j)} + C^{up}(u_n)\left[C^{pp}(u_n) + h^{-1}\Gamma\right]^{-1}(y_{n+1}^{(j)} - G(u_n^{(j)})), \quad j=1,\ldots,J
\end{equation}
for each particle $u^{(j)}$, $j-1,\ldots,J, \ J\in\mathbb N$ in the $n$-th iteration,}
where
\begin{equation}
\label{eq:perturbedobservation}
y_{n+1}^{(j)} = y + \xi_{n+1}^{(j)}
\end{equation}
with $\xi_{n+1}^{(j)}\sim N(0, h^{-1}\Gamma)$, is the artificially perturbed observation and the operators\footnote{The ``$p$'' in the superscript is notation from \cite{iglesias2013ensemble} where quantities in parameter space are called $u$ and objects in observation space are called $p$.} $C^{pp}$ and $C^{up}$ are defined for some $u = (u^{(j)})_{j=1}^J$, with each $u^{(j)}\in X$, as 
\begin{align}
C^{pp}(u) &= \frac{1}{J}\sum_{j=1}^J (G(u^{(j)}) - \overline G) \otimes (G(u^{(j)})-\overline G),\\
C^{up}(u) &= \frac{1}{J}\sum_{j=1}^J (u^{(j)} - \overline u) \otimes (G(u^{(j)})-\overline G),\\
\overline u &= \frac{1}{J}\sum_{j=1}^J u^{(j)},\quad \overline G = \frac{1}{J}\sum_{j=1}^J G(u^{(j)}).
\end{align}

{Equivalently, we can rewrite the iteration in the form}

\begin{equation}
\label{eq:EnKFScheme}
\begin{split}
u_{n+1}^{(j)} = u_n^{(j)} + &hC^{up}(u_n)\left[hC^{pp}(u_n) + \Gamma\right]^{-1}(y- G(u_n^{(j)}))\\ + &\sqrt{h}C^{up}(u_n)\left[hC^{pp}(u_n) + \Gamma\right]^{-1}\Gamma^{1/2}\zeta_{n+1}^{(j)}, \quad j=1,\ldots,J,
\end{split}
\end{equation}
with standard Gaussians $\zeta_{n+1}^{(j)}\sim N(0,1)$ independent and identically distributed. 
Formally, in the limit $h \to 0$ one can see  (cf. \cite{schillings2016analysis}) that this is some non-standard approximation scheme for the following system of SDEs:
\begin{equation}\label{eq:SDEEnKF}
\d u^{(j)} = C^{up}(u)\Gamma^{-1}(y- G(u^{(j)}))\d t + C^{up}(u)\Gamma^{-1/2} \d W^{(j)}, \quad j=1,\ldots,J.
\end{equation}
Due to the derivation, we are particularly interested in the dynamics on $t\in[0,1]$ \rev{(with $t=0$ corresponding to the prior measure and $t=1$ denoting the posterior)}, but we will work with a general bounded time domain $t\in[0, T]$ from here on. 
A well studied example for an approximation scheme is the Euler-Maruyama discretization, which in our example is  of the form
\begin{equation}
\label{eq:EnKFEMScheme}
\tilde u_{n+1}^{(j)} = \tilde u_n^{(j)} + hC^{up}(\tilde u_n)\Gamma^{-1}(y- G(\tilde u_n^{(j)})) + \sqrt{h}C^{up}(\tilde u_n)\Gamma^{-1/2}\zeta_{n+1}^{(j)}.
\end{equation}
In contrast to the usual direction of thought when working with SDEs and discretizations of their dynamics 
(i.e. start with a dynamics and construct a discrete numerical scheme with ``good'' approximation properties), 
we are travelling in the opposite direction: The EnKF methodology yields a numerical iteration which looks like some approximation scheme for an SDE,
and it is interesting to understand whether this numerical scheme is in fact well approximated by the SDE. 

{Schillings and Stuart \cite{schillings2016analysis} formally derived the limiting SDE and focused on the analysis of the dynamical behavior of the particles establishing convergence results for the fixed ensemble size limit. \rev{Motivation of this manuscript is} the question of what nature this continuum limit is and the convergence behavior of the numerical scheme to the limit. This problem is not only of interest for the analysis of the EnKF, but also possibly opens up the perspective to use the EnKF as a numerical discretization scheme for SDEs with non-globally Lipschitz continuous nonlinearities.}

Note at this point that the Euler-Maruyama discretization { is not a suitable approximation of the limiting SDE}. Although for a given realization it converges \rev{pathwise}, 
it was shown \cite{hutzenthaler2011strong} to diverge in the strong sense (i.e. moments of the difference between numerical scheme 
and the solution of the SDE will explode in the limit $h\to 0$). This is mainly due to the fact that the SDE does not have a globally Lipschitz continuous drift and diffusion term. 
 \revcls{At this point, we only manage to carry out the necessary analysis for an extremely simplified toy model exhibiting (arguably) a similar structure as the stochastic differential equation \eqref{eq:SDEEnKF} arising in the EnKF context.} \revcls{The simplified model is a good starting point for the analysis and will form the basis for the more general setting, which will be subject of future work. Due to the subspace property of the EnKF, the dynamics can be described in the finite dimensional coordinate system of the particles, thus the analysis of the finite dimensional case can be straightforwardly generalized to the EnKF setting.
However, the interaction of the particles and the nonlinearity of the forward response operator results in more complicated nonlinearities of the drift and diffusion term, so that we anticipate the need of additional assumptions on their stability in the more general setting.
}

\rev{
\begin{rem}
Note that we when we say ``continuum limit'', we mean the following: 
We artificially augment the state space with an artificial (discrete) time and consider the limit for continuous artific\revcls{i}al time. 
This is not to be confused with either the continuous-time limit in the data assimilation setting (where time is explicitly present in the data) and also not with the mean-field approach 
where the limit $J\to\infty$ (a continuum of particles) is considered. For the latter, see for example \revcls{\cite{le2009large,mandel2011convergence, law2016mean}}.
We also do not comment on fully interacting infinite-dimensional PDE system" and "mean-field limiting system" which also treat the case of $J\to \infty$,
while we keep $J$  fixed and possibly even quite small.
\end{rem}
}

\rev{\paragraph{Constructing a toy model} 
As the full EnKF iteration \eqref{eq:EnKFScheme} is \rev{too} difficult to analyze, 
we make a series of simplifying steps in order to reduce the complexity of the problem, hereby deriving a toy model on which to test the tools we plan on eventually applying \rev{to} the full EnKF setting.

We consider a linear map $G$ and set the dimension of the parameter space to one, i.e. $d=1$ and the number of particles to two, i.e. $J=2$. 
Note that in this case 
\[u^{(1)}-\bar u = -(u^{(2)}-\bar u) =: q.\]
Now the SDE for the EnKF is
\begin{displaymath}
\d u^{(j)} = \frac{1}{2} \sum_{k=1}^2 (u^{(k)}-\bar u) \cdot (Gu^{(k)} -G\bar u) \cdot \left[(y-Gu^{(j)})\d t + \Gamma^{1/2}\d W^j\right]
\end{displaymath}
for $j=1,2$. The SDE for the particle mean is then
\begin{displaymath}
\d \bar u = \frac{1}{2} \sum_{k=1}^2 (u^{(k)}-\bar u) \cdot (Gu^{(k)} -G\bar u) \cdot \left[(y-G\bar u)\d t + \Gamma^{1/2}\d \bar W\right]
\end{displaymath}
and the equation for the particles distance to the mean is 
\begin{displaymath}
\d q = -q (Gq)^2 \d t + q\cdot Gq \d B 
\end{displaymath}
(note that $(u^{(k)}-\bar u) \cdot (Gu^{(k)} -G\bar u) = q \cdot Gq$ for $k=1,2$)
by setting $B = W^{(1)} - \bar W = -(W^{(2)}-\bar W)$. 
For the identity map $G(u) = u$ we get exactly our toy model\eqref{eq:du1}. 
This means that our toy model is a simplification of the dynamics of the particles' distance to their joint mean.
To recover the dynamics of the particles, we can use 
\begin{displaymath}
\d \bar u = q\cdot Gq \cdot (y- G\bar u)\d t + q\cdot Gq \cdot \Gamma^{1/2}\d \bar W,
\end{displaymath}
which we do not consider \rev{here} in our further analysis.}

{
To see how the discretization scheme \eqref{eq:uScheme} arises, we start from \eqref{eq:EnKFScheme} and use the same simplifying assumptions (i.e. $G$ the identity map, one-dimensional parameter space, two particles, $\Gamma = \gamma = 1$). This yields
\begin{align*}
u_{n+1}^{(1)} &= u_n^{(1)} + \frac{\frac{1}{2}\sum_{k=1}^2 (Gu_n^{(1)}-G\bar u_n)\cdot (u^{(1)}-\bar u)}{h\frac{1}{2}\sum_{k=1}^2 (Gu_n^{(1)}-G\bar u_n)^2 + \gamma^2} \cdot \left(h \cdot (y - Gu_n^{(1)}) + \sqrt{h} \cdot  \xi_{n+1}^{(1)}\right)\\
&= u_n^{(1)} + \frac{ q_n^2}{h q_n^2 + 1} \cdot \left(h \cdot (y - Gu_n^{(1)}) + \sqrt{h} \cdot  \xi_{n+1}^{(1)}\right)\\
\bar u_{n+1} &= \bar u_n + \frac{ q_n^2}{h q_n^2 + 1} \cdot \left(h \cdot (y - \bar Gu_n) + \sqrt{h} \cdot  \bar \xi_{n+1}\right)\\
\end{align*}
and thus
\[
q_{n+1} = q_n + \frac{ q_n^2}{h q_n^2 + 1} \cdot \left(h \cdot (-q) + \sqrt{h}\cdot  (\xi_{n+1}^{(1)}-\bar\xi_{n+1})\right),
\]
which is exactly \eqref{eq:uScheme}.

}

%%%%%%%%%%%%%%%%%%%%%%%%%%%%%%%%%%%%%%%%%%%%%
\section{Organization of the paper}

%%%%%%%%%%%%%%%%%%%%%%%%%%%%%%%%5

{The remaining part of the paper is organized as follows.
We present in the section \ref{sec:main} the main results of the paper, the strong convergence of the EnKF scheme to the SDE. 
The proof essentially splits into two parts: the analysis of the error between the solution of the SDE and the solution of the Lipschitz regularized SDE, 
which is presented in section \ref{sec:liperr} and the analysis of the error between the solution of the Lipschitz regularized SDE and its Euler-Maruyama discretization presented in section \ref{sec:lipdisc}. 
The analysis relies on a-priori estimates on the moments of the processes, which are derived in section \ref{sec:moment}.}

\section{Statement of the main results}\label{sec:main}

{We will present in the following the main result, the strong convergence of the EnKF numerical discretization scheme to the SDE.}
The proof of the main statement {essentially} relies {on two lemmata on}
bounding differences (and extracting rates) between stochastic processes up to a deterministic time $T$ 
{in case} that the difference is small up to a stopping time and {the} moments of the stochastic processes are small.
{This approach is based on ideas from Higham, Mao, and Stuart \cite{higham2002strong}.}
\begin{lemma}\label{lem:rate}
{ Let $v_i, i=1,2$ be two stochastic processes with continuous paths and let $e=v_1-v_2$ denote the difference between both processes}. 
 {Further, we denote by $\tau$ }the stopping time defined for $T>0$ and $\gamma_h>0$ by
 \begin{displaymath}
 \tau= T \wedge \inf\{t>0\ :\ |v_1(t)|> \gamma_h^{-1} \text{ or }  |v_{2}(t)|> \gamma_h^{-1} \}\;.
 \end{displaymath}
 {Assuming that for some $C_\star>0$ and $p>0$ it holds}
 \begin{displaymath}
 \mathbb{E}\sup_{[0,T]}|v_1 |^p \leq C_\star \quad \text{ and }\quad   \mathbb{E}\sup_{[0,T]}|v_{2} |^p \leq C_\star \;,
 \end{displaymath}
{and that for some $\delta_h>0$, it holds}
 \begin{displaymath}
   \mathbb{E}\sup_{[0,\tau]}|e |^2 \leq \delta_h^2.
 \end{displaymath}
 Then for $\eta\in(0,p)$ there {exists} a constant $K:= 2^{(p-\eta)/p}  2^{\eta} $ such that
\begin{displaymath}
\mathbb{E}\sup_{[0,T]}|e |^\eta   \leq   K  C_\star \gamma_h^{p-\eta}
  + \delta_h^\eta \;.
\end{displaymath}
\end{lemma}
Thus, for sufficiently small $\eta>0$ the term $\delta_h$ determines the rate of the strong convergence.
If $\delta_h=h^\delta$ and $\gamma_h=h^\gamma$, the effective rate is given by 
\begin{displaymath}\Big(\E\sup_{[0,T]}|e|^\eta\Big)^{1/\eta} \leq \operatorname{Const}\cdot h^{ \min\{ \gamma(p-\eta)/\eta ,  \delta  \}} \;.
\end{displaymath}
{It is easy to check, that the optimal rate is $\delta$ if $\eta \leq p\gamma/(\gamma+\delta)$ \;.}
\begin{proof}
{
For the stopping time, it holds that}
\begin{displaymath}
\mathbb{P}( \tau<T ) 
\leq \mathbb{P}( \sup_{[0,T]}|v_1 | > \gamma_h^{-1} \text{ or }  \sup_{[0,T]}|v_{2} | > \gamma_h^{-1})  \;.
\end{displaymath}
Thus, {Chebychev's inequality yields}
\begin{displaymath}
 \mathbb{P}( \tau<T ) 
\leq 2C_\star \gamma_h^p \;.
\end{displaymath}
{It directly follows that}  
\begin{displaymath} \mathbb{E}\sup_{[0,T]}|e |^p  
\leq 
  [ ( \mathbb{E}\sup_{[0,T]}|v_1 |^p )^{1/p} + ( \mathbb{E}\sup_{[0,T]}|v_1 |^p )^{1/p}  ]^p
\leq  2^p  C_\star \;.  
\end{displaymath}
Now we obtain 
\begin{align*}
  \mathbb{E}\sup_{[0,T]}|e |^\eta   
  = &  \int_{\tau<T} \sup_{[0,T]}|e |^\eta d\mathbb{P} + \int_{\tau=T} \sup_{[0,T]}|e |^\eta d\mathbb{P} \\
  = &  \int_{\tau<T} \sup_{[0,T]}|e|^\eta d\mathbb{P} + \int_{\tau=T} \sup_{[0,\tau]}|e |^\eta d\mathbb{P} \\ 
  \leq & \Big(  \int_{\tau<T}  d\mathbb{P}  \Big)^{(p-\eta)/p} \Big(  \sup_{[0,T]}|e|^p d\mathbb{P}  \Big)^{\eta/p}
  + \mathbb{E} \sup_{[0,\tau]}|e |^\eta  \\ 
  \leq & \left(  2{C_\star}\gamma_h^p \right)^{(p-\eta)/p} \left( 2^p C_\star \right)^{\eta/p}
  + \Big( \mathbb{E} \sup_{[0,\tau]}|e |^2 \Big)^{\eta/2}\\
   \leq &   2^{(p-\eta)/p}  C_\star 2^{\eta}   \cdot \gamma_h^{p-\eta}
  + \delta_h^\eta \;.
\end{align*}

\end{proof}

\begin{corollary}\label{cor:rate}
Given the assumptions of lemma \ref{lem:rate}, $\gamma_h = h^\gamma$ 
and additionally $\delta_h^2 = h^{1 - \rho\gamma}$ for some $\rho>0$, we obtain
\begin{displaymath}
\Big(\E\sup_{[0,T]}|e|^\eta\Big)^{1/\eta} 
\leq C h^{\frac{1}{2}\cdot\frac{p-\eta}{p-\eta + \eta \rho/2} }, 
\end{displaymath}
in particular for $\eta \to 0$ we recover a rate of $h^{1/2}$.
\end{corollary}
\begin{proof}
From lemma \ref{lem:rate} we get
\begin{displaymath}
\mathbb{E}\sup_{[0,T]}|e |^\eta   \leq   K  \gamma_h^{p-\eta}
  + \delta_h^\eta  \leq C \cdot (h^{\gamma(p-\eta)} + h^{\frac{\eta}{2}(1-r\gamma)}).
\end{displaymath}
{
The rate is optimal if $\gamma = \frac{\eta}{2p+\eta(r-2)}$, i.e. }
\begin{displaymath}
\mathbb{E}\sup_{[0,T]}|e |^\eta  \leq C \cdot h^{\frac{\eta}{2}\frac{p-\eta}{p-\eta+\eta r/2}}\,.
\end{displaymath}
{Taking $\eta$-root on both sides yields the claim.}
\end{proof}
\begin{lemma}\label{lem:rate_weak}
{ Let $v_i, i=1,2$ be two stochastic processes with continuous paths and let $e=v_1-v_2$ denote the difference between both processes}. 
 {Further, we denote by $\tau$ }the stopping time defined for $T>0$ and \rev{$\gamma>0$} by
 \begin{displaymath}
 \tau= T \wedge \inf\{t>0\ :\ |v_1(t)|> h^{-\gamma} \text{ or }  |v_{2}(t)|> h^{-\gamma} \}\;.
 \end{displaymath}
 Suppose for some $C_\star>0$ and $p>0$
 \begin{displaymath}
 \mathbb{E}\sup_{[0,T]}|v_1 |^p \leq C_\star \quad \text{ and }\quad   \mathbb{E}\sup_{[0,T]}|v_{2} |^p \leq C_\star 
 \end{displaymath}
 and for $s>p$
 \begin{displaymath}
 \sup_{[0,T]}\mathbb{E}|v_1 |^s \leq C_\star \quad \text{ and }\quad   \sup_{[0,T]}\mathbb{E}|v_{2} |^s \leq C_\star \;.
 \end{displaymath}
 Moreover, {we assume that} 
 \begin{displaymath}
   \mathbb{E}\sup_{[0,\tau]}|e |^2 \leq h^{2\delta}
 \end{displaymath}
 for some $\delta>0$.
 Then,
\begin{displaymath} \sup_{t\in[0,T]} \E |e(t)|^q \leq  Const\cdot\Big[h^{(\delta - \frac{\gamma p}{2})\cdot q} + h^{\gamma p \frac{s-q}{s}}\Big]
\end{displaymath}
and the effective optimal rate is the minimum of both exponents.

If, in addition, $\delta = \frac{1-\rho\gamma}{2}$ for some $\rho > 0$ such that $\delta > 0$, {then} 

\begin{displaymath}\Big(\sup_{t\in[0,T]} \E |e(t)|^q\Big)^\frac{1}{q} \leq C \cdot h^{\frac{p(s-q)}{2p(s-q) + (p+\rho)qs} }\;,
\end{displaymath}
which amounts to an effective rate of $1/2$ for $q\to 0$.
\end{lemma}
\begin{proof}
{Similarly to the proof of lemma \ref{lem:rate}, we obtain}
\begin{displaymath}
 \mathbb{P}( \tau<T ) 
\leq 2C_\star h^{\gamma p}
\end{displaymath}
and 
\begin{displaymath} \mathbb{E}\sup_{[0,T]}|e |^p  
\leq  2^p  C_\star \;.  
\end{displaymath}
Also, $\{\sup_{t\leq T} |e|^2 \geq \kappa_h^2\} \subset \{\tau < T\} \cup \{\sup_{t\leq \tau} |e|^2 \geq \kappa_h^2\}$, 
where $\kappa_h = h^\kappa$ with a parameter $\kappa > 0$ yet to be determined. 
{Thus,}
\begin{displaymath}\P(\sup_{t\leq T} |e|^2 \geq \kappa_h^2) \leq \P(\tau < T) + \frac{\E \sup_{t\leq \tau} |e|^2}{h^{2\kappa}}\leq 2 C_\star \cdot h^{\gamma p} + h^{2\delta-2\kappa}.
\end{displaymath}
The optimal (i.e. lowest) parameter $\kappa$ is $\kappa = \delta - \frac{\gamma p}{2}$, 
then
\begin{displaymath}\P(\sup_{t\leq T} |e|^2 \geq \kappa_h^2) \leq (2C_\star+1) h^{\gamma p}. 
\end{displaymath}
Then
\begin{align*}
\E |e(t)|^q &= \int \chi_{\{\sup_{t\leq T} |e(t)|^2 \leq \kappa_h^2\}}|e(t)|^q + \int \chi_{\{\sup_{t\leq T} |e(t)|^2 > \kappa_h^2\}}|e(t)|^q\\
&\leq \kappa_h^{q} + \P(\sup_{t\leq T} |e|^2 \geq \kappa_h^2)^{\frac{s-q}{s}} \cdot (\E|e(t)|^s)^\frac{q}{s}\;.
\intertext{Now we can take the supremum over time $[0, T]$, and set our optimal parameter $\kappa = \delta - \gamma p/2$ \rev{to obtain}}
\sup_{t\in[0,T]} \E |e(t)|^q &\leq C \cdot[h^{(\delta - \frac{\gamma p}{2})\cdot q} +  h^{\gamma p \frac{s-q}{s}}]\;,
\end{align*}
which proves the first claim. If also $\delta = \frac{1-r\gamma}{2}$, then
\begin{displaymath}\sup_{t\in[0,T]} \E |e(t)|^q \leq  C \cdot [ h^{\frac{1}{2}(1 - r\gamma - \gamma p)\cdot q} + h^{\gamma p \frac{s-q}{s}}] \end{displaymath}
and
we can extract the optimal value for  $ \gamma = \frac{qs}{p(qs + 2(s-q)) + qrs}$, which makes both rates identical. \rev{Thus, we derive}
\begin{displaymath}
\sup_{t\in[0,T]} \E |e(t)|^q \leq C \cdot  h^{q\cdot \frac{p(s-q)}{2p(s-q) + (p+r)qs} }\;.
\end{displaymath}
Taking the $q$-th root yields the claim.
\end{proof}
{We can now present the main result of this paper:}
\begin{theorem}[Strong convergence of the EnKF numerical scheme \eqref{eq:scheme} to the SDE \eqref{eq:du1}]\label{lem:main}
For any {$0 < \alpha < 2$} and $0 < \eta < 1$,
\begin{align*}
\lim_{h=\eps \to 0}h^{-\frac{1}{2}\cdot \frac{3-\alpha}{3+25/3\cdot \alpha}}\cdot \Big( \sup_{t\in [0,T]} \E |\bar v(t) - u(t)|^\alpha\Big)^\frac{1}{\alpha} &= 0\\
\lim_{h=\eps \to 0} h^{-\frac{1}{2}\cdot \frac{1-\eta}{1+3\eta}} \cdot \Big(\E \sup_{t\in [0,T]}|\bar v(t) - u(t)|^\eta\Big)^\frac{1}{\eta} &= 0
\end{align*}
\end{theorem}

\begin{proof}[Idea of Proof]
For the proof we will rely on lemmata \ref{lem:rate}, \ref{lem:rate_weak} and corollary \ref{cor:rate}. 
First we need to verify a-priori estimates for the {approximation} $\vb$, the \rev{\sout{surrogate model} regularization} $v$, and the solution of the SDE $u$, which we will do in section \ref{sec:app}.    
For the error estimate up to a stopping time, we use the triangle inequality $|\vb - u| \leq |\vb - v| + |v - u|$ 
and bound $|v - u|$ in  {lemma \ref{lem:ratesuv}} and $|\vb - v|$ in lemma \ref{lem:ratesvbarv} in the following two sections. 
\end{proof}

\begin{rem}
Note that the bottleneck for $\eta$ is the a-priori bound on $\bar v$. {In case $\E\sup_{t\in[0,T]} |\bar v(t)|^\kappa < C$ for $0 < \kappa < 3/2$}, then $\eta < 1$ {will} get replaced by $\eta < 3/2$ in theorem \ref{lem:main}.
\end{rem}

\section{Bounding the difference between the SDE and its \rev{regularization}}\label{sec:liperr}

{
We denote by $r = u-v$ the difference between the solution of the SDE \eqref{eq:du1} and its \rev{\sout{surrogate model} regularization} \eqref{eq:dv}.
Further, we define the bounded stopping time 
\[\tau = \inf\{t > 0:~ \max \{|u|,|v|\} \geq h^{-\gamma}\}\wedge T\;,\]
i.e. the first time that any of $u$ and $v$ become ``large''. 
This section is devoted to the establishment of bounds on the difference of the SDE and the Lipschitz-regularized version, in particular, we will present rates \rev{w.r.t.}\ the regularization parameter $h=\epsilon$.

For simplicity, we introduce the notation $h^{1-}$ which is supposed to mean $h^{1-\kappa}$ for some $\kappa \in (0,1)$ (and $\kappa$ close to $0$).}

\begin{lemma}\label{lem:ratesuv}
For any $0 < \alpha < 3$ and $0<\beta < 3/2$, {we have:}
\begin{align*}
\lim_{h\to 0} h^{-\frac{1}{2}\frac{3/2-\beta}{3/2+ 2\beta} } \cdot \left(\E \sup_{[t\in [0,T]} |r|^\beta\right)^\frac{1}{\beta} = 0  \\
\lim_{h\to 0} h^{-\frac{1}{2}\frac{3-\alpha}{3+13/2\cdot \alpha}}\cdot \left(\sup_{[t\in [0,T]} \E |r|^\alpha\right)^\frac{1}{\alpha} = 0
\end{align*}
\end{lemma}

\begin{proof}
The proof will follow these steps:
\begin{enumerate}
\item First we use Ito's formula to obtain a bound on moments of $r$ of the form $\E r^2(t\wedge \tau)$. Doing this, we will even obtain a better estimate , i.e.
\begin{equation}\label{eq:goodmoment}
\E r^2(t\wedge \tau)  + \E \int_0^{t \wedge\tau} r^2(s) \cdot \frac{u^2+v^2}{1+h v^2}\d s \leq C\cdot  h^{1-}{\,.}
\end{equation}
\item Next, we use the previous result to get a bound on moments of suprema of $r$, i.e.
\begin{equation}\label{eq:goodmomentsup}
\E \sup_{t\in[0,\tau]}r^2(t) \leq C \cdot h^{1-}{\;.}
\end{equation}
\item Finally, we employ corollary \ref{cor:rate} {to bootstrap our estimates}, which are up to a stopping time. This yields moments up to time $T$ with a rate of strong convergence.
\end{enumerate}

\subsection{Step 1}

If $r = u-v$, and recalling
\begin{displaymath}
\d v = -\frac{v^3}{1+h v^2} \d t + \frac{v^2}{1+hv^2}\d W
\end{displaymath}
and 
\begin{displaymath}
\d u = -u^3 \d t + u^2 \d W,
\end{displaymath}
we have that 
\begin{equation}
\begin{split}
\d r &= -\frac{u^3-v^3}{1+hv^2}\d t - h\frac{u^3v^2}{1+hv^2}\d t + \frac{u^2-v^2}{1+hv^2}\d W + h \frac{u^2v^2}{1+hv^2}\d W \\
&= -r\frac{u^2+uv+v^2}{1+hv^2}\d t - h\frac{u^3v^2}{1+hv^2}\d t + r \frac{u+v}{1+hv^2}\d W + h \frac{u^2v^2}{1+hv^2}\d W \\
&=: -r\cdot T_1(u,v)\d t - h\frac{u^3v^2}{1+hv^2}\d t + r\cdot T_2(u,v) \d W  + h \frac{u^2v^2}{1+hv^2}\d W{\;.}
\end{split}
\end{equation}
{Thus, by Ito's formula}
\begin{align*}
\d r^2 =& 2 r\cdot \d r + (\d r)^2 \\
=& -2{r}^2\cdot T_1 \d t -2h r \frac{u^3v^2}{1+hv^2}\d t + r^2\cdot T_2^2\d t +2h\cdot r \cdot T_2 \frac{u^2v^2}{1+hv^2}\d t\\& + h^2\cdot \frac{u^4v^4}{(1+hv^2)^2}\d t + 2{r}^2\cdot T_2 \d W + 2h\cdot r \cdot \frac{u^2 v^2}{1+hv^2}\d W.
\end{align*}
In integral form this is
\begin{align*}
r^2(s\wedge \tau) &= -\int_0^{s\wedge \tau} r^2\cdot (2T_1-T_2^2) \d t \\
& + h\cdot \left[ \int_0^{s\wedge \tau} -2 r \frac{u^3v^2}{1+hv^2}+2\cdot r \cdot T_2 \frac{u^2v^2}{1+hv^2} + h\cdot \frac{u^4v^4}{(1+hv^2)^2}\d t\right]\\
& + \int_0^{s\wedge \tau} 2e^2\cdot T_2 \d W + h\cdot \int_0^t 2 r \cdot \frac{u^2 v^2}{1+hv^2}\d W{\;.}
\end{align*}
Now
\begin{align*}
2T_1 +T_2^2 = & \frac{2u^2+2uv+2v^2}{1+hv^2}-\frac{(u+v)^2}{(1+hv^2)^2}\\
\geq & \frac{2u^2+2uv+2v^2}{(1+hv^2)^2}-\frac{u^2+2uv+v^2}{(1+hv^2)^2} = \frac{u^2+v^2}{(1+hv^2)^2}\geq 0
\end{align*}
and thus we know that the first integral is negative. We can summarize the part in square brackets by noting that $\tau \leq T$ and that the integrand is bounded by some power of $h$, 
more accurately $h^{-6\gamma}$ (because $u,v,r,T_2$ all are bounded as we only integrate up to stopping time $\tau$). Finally, we get
\begin{equation} \label{eq:control_e2}
r^2(s\wedge \tau) + \int_0^{s\wedge\tau} r^2\cdot \frac{u^2+v^2}{1+hv^2}\d t \leq h^{1-6\gamma} \cdot C\cdot T + \int_0^{s\wedge \tau} 2e^2\cdot T_2 \d W + h\cdot \int_0^{s\wedge \tau} 2 r \cdot \frac{u^2 v^2}{1+hv^2}\d W
\end{equation}
and thus
\begin{align*}
\E r^2(s\wedge \tau) + \E \int_0^{s\wedge\tau} r^2\cdot \frac{u^2+v^2}{1+hv^2}\d t &\leq h^{1-6\gamma} \cdot C\cdot T
\end{align*}
which is \eqref{eq:goodmoment}.

\subsection{Step 2}

Going one step back to \eqref{eq:control_e2} (where we drop for now the second positive term on the left hand side), we take a look at the supremum of $r^2$, of which we know now
\begin{align*}
\sup_{t\in [0,\tau]} r^2(t) &\leq h^{1-6\gamma} \tilde C + 2 \sup_{t\in [0,\tau]} \int_0^{t} r^2\cdot T_2 \d W\\
& + 2h \sup_{t\in [0,\tau]} \int_0^{s\wedge \tau} r \cdot \frac{u^2 v^2}{1+hv^2}\d W{\;.}
\intertext{Applying the expectation, we have}
\E \sup_{t\in [0,\tau]} r^2(t) &\leq h^{1-6\gamma} \tilde C + 2\E \sup_{t\in [0,\tau]} \int_0^{t} r^2\cdot T_2 \d W \\
&+ 2h \E \sup_{t\in [0,\tau]} \int_0^{s\wedge \tau} r \cdot \frac{u^2 v^2}{1+hv^2}\d W{\;.}
\intertext{The last term's integrand {can be} bound again by some power of $h$ as done before for the deterministic integral. Then we just have an expectation of the supremum of the Brownian motion up to a bounded stopping time (which is a constant). The first term behaves nicely under $h\to 0$ so all we need to take care of is the middle term. For this we need the Burkholder-Davis-Gundy inequality:} 
{\E \sup_{t\in [0,\tau]} \int_0^{t} r^2\cdot T_2 \d W} &\leq C \E \left[\int_0^\tau r^4\cdot T_2^2 \d t\right]^\frac{1}{2}
\intertext{if the right hand side is well-defined (i.e. finite). We show that it indeed is finite by further bounding}
{\E \sup_{t\in [0,\tau]} \int_0^{t} r^2\cdot T_2 \d W}&\leq C \E\left[ \sup_{t\in[0,\tau]} r^2(t)\cdot \int_0^\tau r^2\cdot T_2^2 \d t\right]^\frac{1}{2}
\intertext{and an application of the Cauchy-Schwarz inequality yields}
{\E \sup_{t\in [0,\tau]} \int_0^{t} r^2\cdot T_2 \d W}&\leq C\left(\E \sup_{t\in[0,\tau]} r^2(t)\right)^\frac{1}{2}\cdot \left(\E \int_0^\tau r^2\cdot T_2^2 \d t\right)^\frac{1}{2}\;.
\intertext{Observe that the left term is what we started with. We do not worry that the term might be infinite because we only go up to a stopping time. We rather want to show that this term goes to $0$ for $h\to 0$. Inserting the just obtained bound above, we get}
\E \sup_{t\in [0,\tau]} r^2(t) &\leq  h^{1-6\gamma} \tilde C + 2C\left(\E \sup_{t\in[0,\tau]} r^2(t)\right)^\frac{1}{2}\cdot \left(\E \int_0^\tau r^2\cdot T_2^2 \d t\right)^\frac{1}{2} \\
&+ 2C h^{1-5\gamma}\;.
\intertext{The second expectation on the right hand side can be bounded as follows:}
\E \int_0^\tau r^2\cdot T_2^2 \d t & = \E \int_0^\tau r^2 \cdot \frac{(u+v)^2}{(1+hv^2)^2}\d t \leq 2 \E \int_0^\tau \frac{u^2+v^2}{(1+hv^2)^2}\cdot r^2\d t \leq C h^{1-6\gamma} 
\intertext{where the last inequality comes from \eqref{eq:goodmoment}. All in all,}
\E \sup_{t\in [0,\tau]} r^2(t) &\leq  C\cdot h^{1-6\gamma}  + 2C\cdot h^{\frac{1}{2}-3\gamma}\cdot \left(\E \sup_{t\in[0,\tau]} r^2(t)\right)^\frac{1}{2}\cdot 
\intertext{This is equivalent to finding a bound on $A$ given that $A \leq Ch^{1-\kappa} + C'\sqrt{A} h^{\frac{1-\kappa}{2}}$, which yields $A\leq C''\cdot h^{1-\kappa}$ and thus for our problem,}
\E \sup_{t\in [0,\tau]} r^2(t) &\leq  C \cdot h^{1-6\gamma},
\end{align*}
which is \eqref{eq:goodmomentsup}.

\subsection{Step 3}

Now we employ corollary \ref{cor:rate} with $v_1 = u$, $v_2 = v$, $p = 3/2$, $\rho = 6$. \rev{This yields}
\begin{displaymath} \Big(\E \sup_{[t\in [0,T]} |r|^\beta\Big)^\frac{1}{\beta} \leq C h^{\frac{1}{2}\frac{3/2-\beta}{3/2+ 2\beta} }\;. 
\end{displaymath}
For the slightly weaker convergence condition we can apply lemma \ref{lem:rate_weak} with $v_1 = u$, $v_2 = v$, $p=3/2$, $s = 3$ and $\rho=6$\rev{, in order to obtain}
\begin{displaymath}\Big(\sup_{[t\in [0,T]} \E |r|^\alpha\Big)^\frac{1}{\alpha} \leq C h^\frac{3-\alpha}{6+13\alpha}\;.
\end{displaymath}

\end{proof}

\section{Bounding the difference between the \rev{regularization} and its Euler-Maruyama discretization}\label{sec:lipdisc}

We define $e = \bar v - v$, {i.e. we consider now the difference between the \rev{\sout{surrogate model} regularization} \eqref{eq:dv} and its Euler-Maruyama discretization \eqref{eq:scheme}. }The main result of this section is the following lemma. The whole section is devoted to its proof. 
The idea of proof is based on the  
a-priori bounds in section \ref{sec:app} and the application of  lemmata \ref{lem:rate}, \ref{lem:rate_weak} and corollary \ref{cor:rate}. 
\begin{lemma}\label{lem:ratesvbarv}
For any $\eta \in (0,1)$ and $\alpha\in(0, 3)$, {we have}
\begin{align*}
\lim_{h\to 0} h^{-\frac{1}{2}\cdot \frac{1-\eta}{1+3\eta}}\cdot  \left(\E\sup_{t\in [0,T]} |e(t)|^\eta \right)^\frac{1}{\eta} = 0 \\
\lim_{h\to 0}h^{-\frac{1}{2}\cdot \frac{3-\alpha}{3+25/3\cdot \alpha}} \cdot \left( \E  \sup_{t\in[0,T]}|e(t)|^\alpha\right)^\frac{1}{\alpha} &= 0
\end{align*}
\end{lemma}

\subsection{A-priori bounds for the error}

We can write, defining $\eta(t) =k$ and $\eta^+(t) = {k+1}$ for $t\in [t_k, t_{k+1})$.
\begin{align*}
\bar v(t) &= v_0 + \int_{0}^{t} f(\bar v(s))\d s + \int_0^t \sigma(\bar v(s))\d W \\
&+ \underbrace{\sum_{k=1}^{\eta^+(t)}\int_{t_{k-1}}^{t\wedge t_k}f(v_k)-f(\bar v(s))\d s + \int_{t_{k-1}}^{t\wedge t_k}\sigma(v_k)-\sigma(\bar v(s))\d W_s}_{\Res(t)}\;.
\end{align*}
Now, 
\begin{displaymath}\d \Res(t) = [f(v_{\eta(t)})-f(\bar v(t))]\d t + [\sigma(v_{\eta(t)})-\sigma(\bar v(t))]\d W_t. \end{displaymath}
Next, with the error between solution and interpolated scheme $e=\bar v - v$,
\begin{displaymath}
\d e = [f(v+e)-f(v)]\d t + [\sigma(v+e)-\sigma(v)]\d W + \d \Res(t)
\end{displaymath}
and
\begin{align*}
\frac{1}{2}\d |e|^2 &= [f(v+e)-f(v)]e \d t + [\sigma(v+e)-\sigma(v)]e\d W + e\cdot \d\Res \\
 &+ \frac{1}{2}[\d\Res]^2  + \frac{1}{2}[\sigma(v+e)-\sigma(v)]^2\d t + [\sigma(v+e)-\sigma(v)]\d W \d\Res.
\end{align*}
A calculation shows that

\begin{equation}
\begin{split}
&[f(v)-f(w)]\cdot (v-w) + \frac{1}{2}[g(v)-g(w)]^2 \\
&= -\frac{(v-w)^2}{2}\cdot \frac{v^2+w^2+\eps v^2w^2+(v+w)^2\cdot \left(1-\frac{1}{(1+\eps v^2)(1+\eps w^2)}\right)}{1+\eps v^2+\eps w^2 + \eps^2 v^2 w^2}
\end{split}
\end{equation}

and using this above yields
\begin{align*}
\frac{1}{2}\d |e|^2 &= -\frac{|e|^2}{2}\cdot T(v, e)\cdot \d t + [\sigma(v+e)-\sigma(v)]e\d W\\
& + e\cdot[f(v_{\eta(t)})-f(\bar v(t))]\d t + e\cdot [\sigma(v_{\eta(t)})-\sigma(\bar v(t))]\d W  + \\
&+ \frac{1}{2}[\sigma(v_{\eta(t)})-\sigma(\bar v(t))]^2\d t +[\sigma(v+e)-\sigma(v)][\sigma(v_{\eta(t)})-\sigma(\bar v(t))]\d t
\end{align*}
with $T(v,e) \geq 0$ defined as 
\begin{align*}T(v, e) &= \frac{v^2+(v+e)^2+\eps^2v^2(v+e)^2+(v+(v+e))^2\cdot \left(1-\frac{1}{(1+\eps v^2)(1+\eps (v+e)^2)}\right)}{1+\eps v^2+\eps (v+e)^2 + \eps^2 v^2 (v+e)^2}\\
&\geq C\tilde T(v, \bar v)\end{align*}
where $\min\{ \frac1\epsilon;\ \bar{v}^2+v^2 \}=: \tilde T(v, \bar v)$ and the inequality holds from calculation \ref{calc:T}.

In proving the convergence of the scheme to $v$, 
we go very similar steps 1-4 as in the previous section 
(but this time for $e=$ the discrete error instead of the difference between both solutions).
We know 
\[
|\bar v(t) - v_n| \leq h\cdot |f(v_n)| + |W(t)-W(t_n)| \cdot \sigma(v_n)
\] 
and thus
\begin{equation}
|\bar v(t) - v_n|^2 \leq Ch^2|f(v_n)|^2 + C|W(t)-W(t_n)|^2\cdot |\sigma (v_n)|^2.
\end{equation}
\rev{I}n particular,
\begin{equation}\label{eq:eq:approx_interp}
\E|\bar v(t) - v_n|^2 \leq Ch^2 \E |f(v_\eta)|^2 + Ch \E| \sigma(v_\eta)|^2\;.
\end{equation}
With this, \rev{we derive}
\begin{align*}
\tilde T(v_\eta, \bar v) & = \min\{\eps^{-1}, |v_\eta|^2+|\bar v|^2\} \leq C \min\{\eps^{-1}, |v_\eta-\bar v|^2+|\bar v|^2\} \\
&\leq \tilde T(v,\bar v) + \min\{\eps^{-1}, Ch^2|f(v_n)|^2 + C|W(t)-W(t_n)|^2\cdot \sigma (v_n)\} \;.
\end{align*}
We use this in our expression for $\d e^2$ to get rid of the term $\delta \tilde T(v_\eta, \bar v) |e|^2\d t$: 
\begin{align*}
\frac{1}{2}\d |e|^2 &\leq  -\frac{|e|^2}{2}\cdot \tilde T(v, \bar v)\cdot \d t + \delta \tilde T(v_\eta, \bar v) |e|^2\d t + \delta \tilde T(\bar v, v)|e|^2\d t \\
&+ 2C_\delta \tilde T(v_\eta, \bar v) |v_\eta - \bar v|^2 \d t + \frac{1}{2}\tilde T(v_\eta, \bar v) |v_\eta-\bar v|^2 \d t \\
&+ (\sigma(v+e)-\sigma(v))\cdot e \; \d W + (\sigma(v_\eta)-\sigma(v+e))\cdot e \; \d W \\
&\leq -\frac{|e|^2}{2}\cdot \tilde T(v, \bar v)\cdot \d t \\
&+ \delta \tilde T(v, \bar v) |e|^2\d t \\
&+ \delta \min\{\eps^{-1}, Ch^2|f(v_n)|^2 + C|W(t)-W(t_n)|^2\cdot |\sigma (v_n)|^2\} |e|^2\d t \\
& + \delta \tilde T(\bar v, v)|e|^2\d t \\
&+ 2C_\delta \tilde T(v_\eta, \bar v) |v_\eta - \bar v|^2 \d t + \frac{1}{2}\tilde T(v_\eta, \bar v) |v_\eta-\bar v|^2 \d t \\
&+ (\sigma(v+e)-\sigma(v))\cdot e \; \d W + (\sigma(v_\eta)-\sigma(v+e))\cdot e \; \d W \;.
\intertext{Now for small enough $\delta$ we can bring all terms of the form $C\cdot  \tilde T(v, \bar v) |e|^2\d t$ as a positive summand to the left hand side because of the negatively dominating first term:}
\frac{1}{2}\d |e|^2 &+ \frac{1}{2}(1-4\delta) \cdot |e|^2\cdot \tilde T(v, \bar v)\d t \\
&\leq \delta \min\{\eps^{-1}, Ch^2|f(v_n)|^2 + C|W(t)-W(t_n)|^2\cdot |\sigma (v_n)|^2\} |e|^2\d t \\
&+ 2C_\delta \tilde T(v_\eta, \bar v) |v_\eta - \bar v|^2 \d t + \frac{1}{2}\tilde T(v_\eta, \bar v) |v_\eta-\bar v|^2 \d t \\
&+ (\sigma(v+e)-\sigma(v))\cdot e \; \d W + (\sigma(v_\eta)-\sigma(v+e))\cdot e \; \d W \;.
\intertext{And merging similar terms,}
&\leq \delta \min\{\eps^{-1}, Ch^2|f(v_n)|^2 + C|W(t)-W(t_n)|^2\cdot |\sigma (v_n)|^2\} |e|^2\d t \\
&+ C_\delta \tilde T(v_\eta, \bar v) |v_\eta - \bar v|^2 \d t \\
&+ (\sigma(v+e)-\sigma(v))\cdot e \; \d W + (\sigma(v_\eta)-\sigma(v+e))\cdot e \; \d W \;.
\end{align*}

\subsection{Bounds up to the stopping time}

Now  for $M$ fixed later we set
\begin{displaymath}
 \tau = T \wedge \inf\{t > 0: |v(t)|\geq M\vee |\bar v(t)| \geq M \vee |v_\eta(t)|\geq M\}\;.
\end{displaymath}
 Up to this stopping time, $|f(v_\eta)|\leq M^3$, $|\sigma(v_\eta)|\leq M^2$ and $|\tilde T(v_\eta, \bar v)|\leq CM^2$, as well as (from \eqref{eq:eq:approx_interp}
\begin{equation}
\E|\bar v(t) - v_n|^2 \leq ChM^6\;.
\label{eq:numericalerror_stopped}
\end{equation}
Then
\begin{equation}\label{eq:esqd}
\begin{split}
\frac{1}{2}|e(\tau)|^2 &+ \frac{1}{2}(1-4\delta) \cdot \int_0^\tau |e(t)|^2\cdot \tilde T(v, \bar v)\d t \leq C\delta h^2 \cdot M^6\int_0^\tau |e(t)|^2\d t\\
&+ C\delta M^4 \cdot \int_0^\tau |W(t)-W(t_n)|^2\cdot |e(t)|^2\d t + C_\delta M^2\cdot \int_0^\tau |v_\eta - \bar v|^2\d t\\
&+  \int_0^\tau(\sigma(v+e)-\sigma(v))\cdot e \; \d W + \int_0^\tau(\sigma(v_\eta)-\sigma(v+e))\cdot e \; \d W 
\end{split}
\end{equation}
or better with a different time argument:
\begin{equation}\label{eq:esqd2}
\begin{split}
\frac{1}{2}|e({s \wedge \tau})|^2 &+ \frac{1}{2}(1-4\delta) \cdot \int_0^{s \wedge \tau} |e(t)|^2\cdot \tilde T(v, \bar v)\d t \leq C\delta h^2 \cdot M^6\int_0^{s \wedge \tau} |e(t)|^2\d t\\
&+ C\delta M^4 \cdot \int_0^{s \wedge \tau} |W(t)-W(t_n)|^2\cdot |e(t)|^2\d t + C_\delta M^2\cdot \int_0^{s \wedge \tau} |v_\eta - \bar v|^2\d t\\
&+  \int_0^{s \wedge \tau}(\sigma(v+e)-\sigma(v))\cdot e \; \d W + \int_0^{s \wedge \tau}(\sigma(v_\eta)-\sigma(v+e))\cdot e \; \d W \;.
\end{split}
\end{equation}
Up to stopping time $\tau$, we can brute-force bound $|e(t)|^2\leq 2M^2$. Second, in order to be able to apply the expectation on both sides in the next step, 
we estimate all positive integrals on the r.h.s. from above by replacing the upper integral boundary $s\wedge \tau$ by its deterministic upper bound $T$:
\begin{align*}
\frac{1}{2}|e({s \wedge \tau})|^2 &+ \frac{1}{2}(1-4\delta) \cdot \int_0^{s \wedge \tau} |e(t)|^2\cdot \tilde T(v, \bar v)\d t \leq C\delta h^2 \cdot M^8T\\
&+ C\delta M^6 \cdot \int_0^{T} |W(t)-W(t_n)|^2\d t + C_\delta M^2\cdot \int_0^{T} |v_\eta - \bar v|^2\d t\\
&+  \int_0^{s \wedge \tau}(\sigma(v+e)-\sigma(v))\cdot e \; \d W + \int_0^{s \wedge \tau}(\sigma(v_\eta)-\sigma(v+e))\cdot e \; \d W \;.
\intertext{We take the expectation which removes the last two integrals and use \eqref{eq:numericalerror_stopped} to obtain}
\frac{1}{2}\E|e({s \wedge \tau})|^2 &+ \frac{1}{2}(1-4\delta) \cdot \E\int_0^{s \wedge \tau} |e(t)|^2\cdot \tilde T(v, \bar v)\d t \\
&\leq C\delta h^2 \cdot M^8T+ hC\delta M^6 T + hC_\delta M^8T\;,
\end{align*}
i.e.,
\begin{equation}\label{eq:boundintegral_e}
\frac{1}{2}\E|e({s \wedge \tau})|^2 + \frac{1}{2}(1-4\delta) \cdot \E\int_0^{s \wedge \tau} |e(t)|^2\cdot \tilde T(v, \bar v)\d t \leq C_\delta h M^8 T\;.
\end{equation}
We go back to \eqref{eq:esqd2}, drop the second term on the left hand side and apply the supremum to get
\begin{align*}
\sup_{s\leq \tau}\frac{1}{2}|e(s)|^2 &\leq C\delta h^2 M^8 T + C\delta M^6 \int_0^T |W(t)-W(t_n)|^2\d t + C_\delta M^2 \int_0^T |v_\eta-\bar v|^2\d t\\
& + \sup_{s\leq \tau}\int_0^{s }(\sigma(v+e)-\sigma(v))\cdot e \; \d W + \sup_{s\leq \tau}\int_0^{s }(\sigma(v_\eta)-\sigma(v+e))\cdot e \; \d W 
\end{align*}
{and thus}
\begin{align*}
\E \sup_{s\leq \tau}\frac{1}{2}|e(s)|^2 &\leq C\delta h^2 M^8 T + C\delta M^6 hT + C_\delta M^8h \\
& + \E\sup_{s\leq \tau}\int_0^{s }(\sigma(v+e)-\sigma(v))\cdot e \; \d W \\
&+ \E\sup_{s\leq \tau}\int_0^{s }(\sigma(v_\eta)-\sigma(v+e))\cdot e \; \d W.
\end{align*}
Using BDG, 
 \begin{align*}
 \E\sup_{s\leq \tau} \int_0^s [\sigma(\bar v)-\sigma(v)]\cdot e \; \d W &\leq \E \left[\int_0^\tau |\sigma(\bar v)-\sigma(v)|^2 |e|^2\d t \right]^\frac{1}{2}\\
 &\leq  \E \left[\int_0^\tau (|\bar v|^2 + |v|^2) |e|^4\d t \right]^\frac{1}{2}\\
 &\leq \sqrt{\E\sup_{t\leq \tau}|e|^2}\cdot \sqrt{\E \int_0^T (|\bar v|^2 + |v|^2) |e|^2\d t}\\
 &\leq \sqrt{\E\sup_{t\leq \tau}|e|^2}\cdot \sqrt{C_\delta M^8 T h} 
 \end{align*}
 where the last step is due to \eqref{eq:boundintegral_e}.
 Similarly,
 \begin{align*}
  \E\sup_{s\leq \tau} \int_0^s [\sigma(v_\eta)-\sigma(\bar v)]\cdot e \; \d W &\leq\sqrt{\E\sup_{t\leq \tau}|e|^2}\cdot \sqrt{\E\int_0^\tau (|v_\eta|^2 + |\bar v|^2)|v_\eta - \bar v|^2 \d t}\\
  &\leq \sqrt{\E\sup_{t\leq \tau}|e|^2} \sqrt{\E\int_0^T CM^2|v_\eta - \bar v|^2\d t}\\
  &\leq \sqrt{\E\sup_{t\leq \tau}|e|^2} \sqrt{CTM^6h}\;.
 \end{align*}
 So, all in all,
 \begin{align*}
 \E \sup_{s\leq \tau}\frac{1}{2}|e(s)|^2 &\leq C_\delta h M^8 T +\sqrt{\E\sup_{s\leq \tau}|e|^2}\cdot C\cdot M^4 h^\frac{1}{2}T^\frac{1}{2}
 \end{align*}
 which yields
 \begin{equation}
 \E \sup_{s\leq \tau}\frac{1}{2}|e(s)|^2 \leq C_\delta h M^8 T\;.
 \end{equation}
 Now we set $M = h^{-\gamma}$ for some $\gamma > 0$ small (at least $\gamma < \frac{1}{8}$). 

This amounts to our previous results being
\begin{equation}
\frac{1}{2}\E|e({s \wedge \tau})|^2 + \frac{1}{2}(1-4\delta) \cdot \E\int_0^{s \wedge \tau} |e(t)|^2\cdot \tilde T(v, \bar v)\d t \leq C_\delta h^{1-8\gamma} T
\end{equation}
 and
 \begin{equation}
 \E \sup_{s\leq \tau}\frac{1}{2}|e(s)|^2 \leq C_\delta h^{1-8\gamma} T {\;.}
 \end{equation}

 \subsection{Conclusion of the proof}
 
Now we can rely on the a-priori estimates of section \ref{sec:app} and use corollary \ref{cor:rate} with $v_1 = v$, $v_2 = \vb$, $p=1$ and $r=8$. 
Thus we get
\begin{displaymath} \Big(\E\sup_{t\in [0,T]} |e(t)|^\eta \Big)^\frac{1}{\eta}   \leq C h^{\frac{1}{2}\cdot \frac{1-\eta}{1+3\eta}}\quad \text{for } \eta < 3/2.
\end{displaymath}
Also, application of lemma \ref{lem:rate_weak} with additionally $s=3$ yields
\begin{displaymath} 
\left(\sup_{t\in [0,T]}\E |e(t)|^\alpha \right)^\frac{1}{\alpha} \leq C h^{\frac{1}{2}\cdot \frac{3-\alpha}{3+25/3\cdot \alpha}}
\qquad {\text{for } \alpha < 2.}
\end{displaymath}

\section{A-priori estimates}\label{sec:moment}
\label{sec:app}

We start with giving a-priori estimates for various versions of momenta of the processes $u$, $v$ and $\vb$. Recall (we have exchanged $\eps$ for $h$)
\begin{align}
\d u &= -u^3 \d t + u^2 \d W,\label{eq:du2}\\
\d v &= -\frac{v^3}{1+h\cdot v^2}\d t + \frac{v^2}{1+h\cdot v^2}\d W,\\
\d \vb &= -\frac{v_n^3}{1 + hv_n^2}\d t + \frac{v_n^2}{1+hv_n^2}\d W,\quad \text{for $t\in [nh, (n+1)h$)}. \label{eq:dvb}
\end{align}
We start with moment bounds on $u$.
\begin{lemma}
For the solution $u$ of the SDE \eqref{eq:du2} \rev{and $T>0$}
\begin{equation}
\sup_{t\in[0,T]} \E |u(t)|^\alpha + \E\int_0^T |u(s)|^{\alpha+2}\d s < C
\end{equation}
for every $0 < \alpha < 3$. Also,
\begin{equation}
\E \sup_{t\in [0,T]} |u(t)|^\beta < C
\end{equation}
for every $0 < \beta < 3/2$.
\end{lemma}
\rev{Note that the constant $C>0$ is allowed to depend on $\alpha$, $\beta$, and $T$.}
\begin{proof}
Ito's formula for $u$ yields 
\begin{equation}
\d u^\gamma = -\gamma\left[1-\frac{\gamma-1}{2}\right]\cdot u^{\gamma +2}\d t + \gamma u^{\gamma +1}\d W
\end{equation}
This means that for the finite stopping time $\tau_n=\inf \{t> 0: u(t) > n\}\wedge T$,
\begin{align*}
\E u(t\wedge \tau_n)^\gamma= -\gamma\left[1-\frac{\gamma-1}{2}\right]\cdot \E\int_0^{t\wedge \tau_n} u^{\gamma+2}(s)\d s + u_0^\gamma
\end{align*}
and
\begin{align*}
\limsup_{n\to\infty}\E u(t\wedge \tau_n)^\gamma&= -\gamma\left[1-\frac{\gamma-1}{2}\right]\cdot \E\int_0^{t} u^{\gamma+2}(s)\d s + u_0^\gamma
\intertext{by monotone convergence (we know that $u\geq 0$ and hence $\int_0^{t\wedge \tau_{n+1}} u^{\gamma+2}(s)\d s \geq \int_0^{t\wedge \tau_{n}} u^{\gamma+2}(s)\d s$). As the left hand side is positive and the first term on the right hand side is negative for $\gamma \in (0,3)$, we can take the supremum to obtain the first claim}
\sup_{t\in[0,T]} \E u(t)^\gamma + \E\int_0^{t} u^{\gamma+2}(s)\d s &< C\;.
\end{align*}
In particular,
\begin{displaymath}\sup_{t\in[0,T]}  \E \int_0^t u^p(s)\d s < C
\end{displaymath}
exists and is bounded for all $p\in (2, 5)$ \rev{with a constant depending on $p$ and $T$}. With this, we obtain the following result \rev{for} the supremum:
\begin{align*}
\E \sup_{[0,\tau_n]} u^\gamma(t) &\leq u_0^\gamma + \E \sup_{[0,\tau_n]} \left[-\gamma\left[1-\frac{\gamma-1}{2}\right]\cdot \int_0^{t} u^{\gamma+2}(s)\d s + \int_0^t \gamma u^{\gamma + 1}\d W(s)\right]\\
&\leq u_0^\gamma + \E \sup_{[0,\tau_n]} \left[\int_0^t \gamma u^{\gamma + 1}\d W(s)\right]\\
\intertext{because the first term in the square brackets is negative\rev{. Now we continue} with Burkholder-Davis-Gundy \rev{inequality}}
&\leq u_0^\gamma + \E \left[\int_0^{\tau_n}\gamma^2 u^{2\gamma +2 }(s)\d s\right]^{\frac{1}{2}}\\
&\leq  u_0^\gamma + \left[\E \int_0^{\tau_n}\gamma^2 u^{2\gamma +2 }(s)\d s\right]^{\frac{1}{2}}
\end{align*}
and the last integral exists for $2\gamma+2 \in (2,5)$, i.e. $\gamma \in (0, 3/2)$. 
\end{proof}

\begin{lemma}\label{lem:aprioriv}
\rev{For $T>0$} 
\begin{equation} 
\sup_{t\in [0,T]} \E |v(t)|^\alpha + \E \int_0^T \frac{v^{\alpha+2}}{(1+hv^2)^2}\d s < C 
\end{equation}
for all $\alpha < 3$ and 
\begin{equation}
\E \sup_{t\in [0,T]}  |v(t)|^\beta < C 
\end{equation}
for all $\beta < 3/2$.
\end{lemma}
\begin{proof}
Ito's formula yields
\begin{align*}
\d v^\alpha &= -\alpha \cdot \left(1-\frac{\alpha-1}{2}\right)\cdot \frac{v^{\alpha+2}}{(1+hv^2)^2}\d t  - \alpha h \cdot \frac{v^{\alpha+4}}{(1+hv^2)^2}\d t + \alpha \frac{v^{\alpha+1}}{1+hv^2}\d W\;.
\end{align*}
When we define a stopping time $\tau_n = \inf\{t>0: |v|> n\}$, we obtain
\begin{equation*}
\begin{split}\E v(t\wedge\tau_n)^\alpha &=  -\alpha \cdot \left(1-\frac{\alpha-1}{2}\right)\cdot \E \int_0^{t\wedge\tau_n}\frac{v^{\alpha+2}}{(1+hv^2)^2}\d s \\
& - \alpha h \cdot \E \int_0^{t\wedge\tau_n} \frac{v^{\alpha+4}}{(1+hv^2)^2}\d s + v_0^\alpha\end{split}\end{equation*}

and for $\alpha \in (0, 3)$ by monotone convergence $(n\to \infty)$ even

\begin{displaymath}\E v(t)^\alpha + \alpha \cdot \left(1-\frac{\alpha-1}{2}\right)\cdot \int_0^{t}\frac{v^{\alpha+2}}{(1+hv^2)^2}\d s  + \alpha h \cdot \int_0^{t} \frac{v^{\alpha+4}}{(1+hv^2)^2}\d s =  v_0^\alpha, \end{displaymath}
and thus bounded uniformly in $h$, also after taking the supremum over $[0,T]$ on both sides.

For the second claim, note that for $\beta < 3$ (see the first line of the proof of the first claim)
\begin{align*}
\E \sup_{t\in [0,T]} v^\beta &\leq \E \sup_{t\in [0,T]}  \int_0^t \frac{v^{\beta+1}}{1+hv^2}\d W \\
&\leq \sqrt{\E \int_0^T \frac{v^{2\beta+2}}{(1+hv^2)^2}\d t} \leq C
\end{align*}
after an application of Burkholder-Davis-Gundy and using the bound on the integral obtained above, which holds for $2\beta + 2 \leq 5$, i.e. $\beta < 3/2$.
\end{proof}
For the discretization we can achieve the same bound on the "$\sup\E$" moment, but a only order $1$ (instead of $3/2$ for $v$) for the "$\E\sup$" moment.

\begin{lemma}\label{lem:apriorivbar}
 For the solution $\vb$ of \eqref{eq:dvb}, we have \rev{for $T>0$}
{
 \begin{equation}
   \sup_{t\in[0,T]}\mathbb{E}  |\vb(t)|^2  
   +   \mathbb{E}  \int_0^t \frac{\vb^4([s]_h)}{(1+h\vb^2([s]_h))^2} ds  \leq C 
\end{equation}  }
 and 
 \begin{equation}
 \mathbb{E}\sup_{[0,T]}|\vb| \leq C
 \end{equation}
\end{lemma}

\begin{proof}
Define $[t]_h = t_n$ such that $t \in [t_n, t_{n+1})$. Note first for $p\in(2,3)$
\begin{displaymath}
(|x|^p )' = p|x|^{p-2}x \qquad\text{and}\qquad (|x|^p )'' = p(p-1)|x|^{p-2}.
\end{displaymath}

By Ito-formula 
\begin{align*}
 d|\vb|^p(t) = 
p \vb(t)|\vb(t)|^{p-2}   \frac{-\vb^3([t]_h)}{1+h\vb^2([t]_h)} dt 
 & + p \vb(t) |\vb(t)|^{p-2} \frac{\vb^2([t]_h)}{1+h\vb^2([t]_h)}dW \\&
+ \frac12p(p-1)|\vb(t)|^{p-2}\frac{\vb^4([t]_h)}{(1+h\vb^2([t]_h))^2} dt\;.
\end{align*}
Thus for $t\in [t_n,t_{n+1})$ \rev{we derive}
\begin{displaymath}
d|\vb|^p(t) =  
p |\vb(t)|^{p-2}
\Big[
 \vb(t)   \frac{-v_n^3}{1+hv_n^2} dt 
+  \vb(t) \frac{v_n^2}{1+hv_n^2}dW 
+ \frac{(p-1)v_n^4}{2(1+hv_n^2)^2} dt\Big]\;.
\end{displaymath}
In particular, for $p=2$,
\begin{displaymath}\partial_t \mathbb{E}  \vb^2(t) 
= 2\mathbb{E}  \vb(t)   \frac{-v_n^3}{1+hv_n^2} 
+  \mathbb{E}\frac{v_n^4}{(1+hv_n^2)^2} \;.
\end{displaymath}
Using the definition of $\vb$ and the independence of the stochastic increment from the filtration at time $t_n$ yields
\begin{align*}
 \partial_t \mathbb{E}  \vb^2(t) 
&= -  2\mathbb{E}   \frac{v_n^4}{1+hv_n^2} 
+ 2(t-t_n) \mathbb{E}  \frac{v_n^6}{(1+hv_n^2)^2} 
+  \mathbb{E}\frac{v_n^4}{(1+hv_n^2)^2} 
\\&
= \mathbb{E}   \frac{-2v_n^4(1+hv_n^2) + 2(t-t_n) v_n^6 + v_n^4}{(1+hv_n^2)^2} 
\\&
=- \mathbb{E}   \frac{v_n^4}{(1+hv_n^2)^2}     + 2 \underbrace{(t-t_n-h)}_{<0} \mathbb{E} \frac{v_n^6}{(1+hv_n^2)^2} \;.
\end{align*}
In particular,
which proves the first claim for $p=2$,
\begin{displaymath}
\mathbb{E}  \vb^2(t)  + \int_0^t \mathbb{E}   \frac{\vb^4([s]_h)}{(1+h\vb^2([s]_h))^2} 
+ 2 \int_0^t ([s]_h+h-s) \mathbb{E} \frac{v^6([s]_h)}{(1+hv^2([s]_h))^2} ds  
\leq \mathbb{E}  \vb^2(0) \;.
\end{displaymath}
  Regarding the second claim, from the definition we know 
\begin{align*}
|\vb(t)| 
&\leq  |\vb(0)|  + \int_0^t \frac{|\vb|^3([s]_h)}{1+h\vb^2([s]_h)} ds +  \Big|\int_0^t \frac{\vb^2([s]_h)}{1+h\vb^2([s]_h)} dW(s)\Big| \\
\\
&\leq  |\vb(0)|  + \Big( \int_0^t|\vb(s)|^2 ds\Big)^{1/2} \Big( \int_0^t \frac{|\vb|^4([s]_h)}{(1+h\vb^2([s]_h))^2} ds \Big)^{1/2} +  \\
&\Big|\int_0^t \frac{\vb^2([s]_h)}{1+h\vb^2([s]_h)} dW(s)\Big|\;.
\end{align*}
Using Burkholder-Davis Gundy, we obtain:
\begin{displaymath}
\mathbb{E}\sup_{[0,T]}|\vb| \leq  \mathbb{E}|\vb(0)|  + \sqrt{T} \mathbb{E}\vb(0)^2  
+ \mathbb{E} \Big(\int_0^T  \frac{\vb^4([s]_h)}{(1+h\vb^2([s]_h))^2} ds \Big)^{1/2} \\
\end{displaymath}
where the boundedness of the last integral holds by setting $p=2$ in the first claim.
\end{proof}

\section{Conclusion} 
{We were able to prove that a simplified version of the numerical
scheme arising from EnKF continuum analysis} is a strongly converging explicit scheme, 
which is a rare and desirable property. {It is remarkable that the Ensemble Kalman methodology canonically 
yields a strongly converging numerical scheme for an SDE with non-globally Lipschitz continuous nonlinearities.  
The analysis presented here was conducted for a simplified model (one-dimensional state and observation space, \rev{two} particle setting, identity operator as forward response operator). However, the analysis suggests that the results can be carried over to a much more general setting, i.e. \eqref{eq:EnKFScheme} seems to be a strongly convergent iteration for a much broader class of SDE \eqref{eq:SDEEnKF}. The generalization of the theory will be subject of future work.}

{The {weak} ``taming'' (i.e. the fact that the numerical scheme divides by $(1 + \eps v^2)$) was necessary to prove the }a-priori moment bounds. This means that the method of proof presented here could in principle point to a method of constructing strongly converging numerical schemes for non-globally Lipschitz SDEs: By constructing a ``Lipschitzified'' version (this would be $v$ in our example) of the SDE (here: $u$) and defining the vanilla Euler-Maruyama discretization ($v_n$) for this \rev{\sout{surrogate model} regularized} SDE, we might sometimes obtain a strongly converging numerical scheme for the original SDE. Note that this is both similar and contrary to the spirit of Hutzenthaler and Jentzen's construction of tamed schemes: They ``tame'' the numerical scheme (or rather its increment term), where in the setting presented here we ``tame'' (or ``Lipschitzify'') the SDE itself. It seems that this can only work for SDEs exhibiting some kind of stability (in this case all solutions tend to stay close to the origin) but further analysis is needed to corroborate this claim.

\section*{Appendix}

A recurring function is 
\begin{displaymath}\tilde T(a, b) = \min\{\eps^{-1}, a^2 + b^2\}.\end{displaymath}
\begin{calc}\label{calc:T}
\begin{align*}
 T(v, e) 
&=  \frac{(v+e)^2+v^2+\eps(v+e)^2v^2+(2v+e)^2\cdot \left(1-\frac{1}{(1+\eps (v+e)^2)(1+\eps v^2)}\right)}{1+\eps (v+e)^2+\eps v^2 + \eps^2 v^2 (v+e)^2} \\
&\geq \frac{(v+e)^2+v^2+\eps(v+e)^2v^2 }{1+\eps (v+e)^2+\eps v^2 + \eps^2 v^2 (v+e)^2} \\
&\geq \frac{\bar{v}^2+v^2+\eps\bar{v}^2v^2 }{1+\eps \bar{v}^2+\eps v^2 + \eps^2 v^2 \bar{v}^2}\\
&\geq C \min\{ \frac1\epsilon;\ \bar{v}^2+v^2 + \eps v^2 \bar{v}^2\} \\
&\geq C \min\{ \frac1\epsilon;\ \bar{v}^2+v^2 = \tilde T(v, \bar v)\}
\end{align*}
\end{calc}

\begin{calc}
\begin{align*}
 |f(\xi)-f(z)|  &= \frac{z^2+z\xi+ \xi^2}{(1+\epsilon z^2)(1+\epsilon \xi^2)}\cdot |z-\xi| \\
 \leq  & \frac12 \frac{z^2+ \xi^2}{1+\epsilon z^2+\epsilon \xi^2}\cdot |z-\xi| \\
 \leq  &  C \min\{ \frac1\epsilon;\ {z}^2+\xi^2 \}  \cdot |z-\xi|
\end{align*}

Thus
\begin{displaymath}
|f(v_\eta)-f(\bar{v})| \leq  \tilde{T}(v_\eta,\bar{v}) |v_\eta-\bar{v}|
\end{displaymath}
\end{calc}
\begin{calc}
\begin{align*}
 |\sigma(\xi)-\sigma(z)|  &= \frac{z+\xi}{(1+\epsilon z^2)(1+\epsilon \xi^2)}\cdot |z-\xi| \\
 &\leq   \sqrt2 \frac{\sqrt{z^2+ \xi^2}}{1+\epsilon z^2+\epsilon \xi^2}\cdot |z-\xi| \qquad \text{by Cauchy-Schwarz}\\
 &\leq    C \min\{ \frac1{\sqrt{\epsilon}};\ \sqrt{ {z}^2+\xi^2} \}  \cdot |z-\xi|\\
 & = \sqrt{\tilde{T}(x,\xi)} \cdot |z-\xi|
\end{align*}
\end{calc}

\section*{Acknowledgments}
P. W. thanks Andrew M Stuart for excellent working conditions in Warwick, Cusanuswerk for additional funding during his stay as well as Weijun Xu for pointing out a mistake. 

\bibliographystyle{siam}
\bibliography{lit}

\end{document}